\documentclass[a4paper,11pt]{amsart}
\usepackage{amssymb,amscd}
\usepackage{mathrsfs} 
\usepackage{tikz-cd}
\usepackage{enumitem}
\usepackage[capitalize]{cleveref}
\newcounter{alph}

\newtheorem{theo}[alph]{Theorem}

\numberwithin{equation}{section}
\newtheorem{cor}[equation]{Corollary}
\newtheorem{lem}[equation]{Lemma}

\newtheorem{thm}[equation]{Theorem}

\theoremstyle{definition}

\newtheorem{exa}[equation]{Example}
\newtheorem{rem}[equation]{Remark}

\def\C{\mathbb C}
\def\F{\mathbb F}

\def\R{\mathbb R}
\def\H{\mathbb H}
\def\O{\mathbb O}

\def\ve{\varepsilon}

\def\la{\langle}
\def\ra{\rangle}

\newcommand{\dive}{\operatorname{div}}

\newcommand{\ds}{\operatorname{ds}}
\newcommand{\dt}{\operatorname{dt}}
\newcommand{\ess}{\operatorname{ess}}
\newcommand{\id}{\operatorname{id}}
\newcommand{\im}{\operatorname{im}}

\newcommand{\tr}{\operatorname{tr}}
\newcommand{\vol}{\operatorname{vol}}

\begin{document}


\title[Differential form spectrum of orbifolds]
{On the differential form spectrum of geometrically finite orbifolds}
\author{Werner Ballmann}
\address
{WB: Max Planck Institute for Mathematics,
Vivatsgasse 7, 53111 Bonn}
\email{hwbllmnn\@@mpim-bonn.mpg.de}
\author{Panagiotis Polymerakis}
\address{PP: Max Planck Institute for Mathematics,
Vivatsgasse 7, 53111 Bonn}
\email{polymerp\@@mpim-bonn.mpg.de}

\thanks{\emph{Acknowledgments.}
We are grateful to the Max Planck Institute for Mathematics
and the Hausdorff Center for Mathematics in Bonn for their support and hospitality.}

\date{\today}


\begin{abstract}
We derive lower bounds for the essential spectrum of the Hodge-Laplacian
on geometrically finite orbifolds and their suborbifolds.
\end{abstract}

\maketitle

\tableofcontents

\section{Introduction}
\label{intro}

The essential spectrum of the Hodge-Laplacian on differential forms
on Riemannian manifolds depends on the geometric structure of the manifolds at infinity.
Geometrically finite manifolds are complete Riemannian manifolds with pinched negative sectional curvature
and restrictions on their geometric structure at infinity.
It is therefore interesting to investigate the influence of these restrictions on the essential spectrum
of the Hodge-Laplacian.
This has been undertaken in several instances, partly only for the Laplacian on functions,
see e.g. \cite{BunkeOlbrich00,Hamenstaedt04,Li20,MazzeoPhillips90}.
Since our methods allow for it, we study the Hodge-Laplacian on differential forms
on geometrically finite orbifolds. 

Let $O$ be a complete and connected Riemannian orbifold with sectional curvature $-b^2\le K_O\le-a^2$.
Then $O$ is a quotient $\Gamma\backslash X$,
where $X$ is a simply connected and complete Riemannian manifold
and $\Gamma$ a properly discontinuous group of isometries of $X$.
For such an $O$,
let $\Omega$ be the complement of the \emph{limit set} $\Lambda$ of $\Gamma$
in the \emph{ideal boundary} $X_\iota$ of $X$.
Following Bowditch \cite{Bowditch95}, we say that $O$ and $\Gamma$ are \emph{geometrically finite}
if $\Gamma\backslash(X\cup\Omega)$ has at most finitely many ends
and each end of $\Gamma\backslash(X\cup\Omega)$ is parabolic (see \cref{secgefi} for details).
A hyperbolic surface is geometrically finite if and only if it is of finite type.
However, the end structure of a geometrically finite orbifold
is, in general, much more complicated than that of surfaces of finite type.

Assume from now on that $O$ is a geometrically finite orbifold of dimension $m$,
and let $F\to O$ be a flat Riemannian vector bundle.
For $k\ge 0$, let $\Delta_k^F$ be the Hodge-Laplacian on $\Lambda^kO\otimes F$,
and denote by $\lambda_k(O,F)$ and $\lambda_k^{\ess}(O,F)$ the bottom of the spectrum
and the essential spectrum of the closure $\bar\Delta_k^F$ of $\Delta_k^F$ on $C^\infty_c(\Lambda^kO\otimes F)$.
Recall that $\lambda_k^{\ess}(O,F)>0$ if and only if $\bar\Delta_k^F$ is a Fredholm operator.

\begin{theo}\label{thmest}
If $(m-k-1)a - kb > 0$, then
\begin{align*}
	\lambda_k^{\ess}(O,F),\lambda_{m-k}^{\ess}(O,F)
	\ge \frac14((m-k-1)a - kb)^2.
\end{align*}
\end{theo}

\cref{thmest} refines \cite[Theorem 3.2]{DonnellyXavier84} in the case considered there,
namely $O=X$, and generalizes {\cite[Theorem B]{BallmannBruening01},
where manifolds $M=\Gamma\backslash X$ of finite volume are discussed.
By \cite[Example 5.5]{BallmannBruening01}, the estimate is sharp.

We say that a Riemannian orbifold $O$ is \emph{hyperbolic}
if it can be written as a quotient $\Gamma\backslash X$,
where $X$ is one of the hyperbolic spaces $H_\mathbb{F}^\ell$ with $\mathbb{F}\in\{\R,\C,\H,\O\}$,
endowed with its canonical Riemannian metric,
which is unique if normalized such that the maximum of its sectional curvature equals $-1$.
Then $m=\dim X=\ell d$, where $d=\dim_\R \mathbb{F}$, and $h_X = (\ell+1)d-2$
is equal to the asymptotic volume growth of $X$,
frequently referred to as the \emph{volume entropy of $X$}.

\begin{theo}\label{thmest2}
Suppose that $O$ is a geometrically finite quotient of $H_\mathbb{F}^\ell$.
Let
\begin{align*}
	d_k =
	\begin{cases}
	(\ell+1)d-2 - 4k &\text{for $0\le k\le d-1$,} \\
	(\ell-1)d - 2k &\text{for $k\ge d-1$,}
	\end{cases}
\end{align*}
and suppose that $d_k>0$.
Then $\lambda_k^{\ess}(O,F),\lambda_{m-k}^{\ess}(O,F)\ge d_k^2/4$.
\end{theo}

If $\ell\ge3$, we have $d_k>0$ for $k<(\ell-1)d/2=(m-d)/2$.
For $\ell=2$ and $\mathbb{F}\in\{\R,\C\}$, only $d_0>0$.
For $\ell=2$ and $\mathbb{F}=\H$, $d_k>0$ for $0\le k\le 2$.
For $\mathbb{F}=\O$, we have $\ell=2$ and $d_k>0$ for $0\le k\le 5$.

Clearly, if $\Omega\ne\emptyset$,
any neighborhood of infinity of $O=\Gamma\backslash H_\mathbb{F}^\ell$
contains balls which are isometric to a ball of radius $r$ in $H_\mathbb{F}^\ell$, for any $r>0$.
Since the pullback of $F$ to $H_\mathbb{F}^\ell$ is a trivial flat bundle,
it follows that the essential spectrum of $\Delta_k^F$
contains the spectrum of $\Delta_k$ on $H_\mathbb{F}^\ell$.
On the other hand,
the contribution of any parabolic end of $O$ to the essential spectrum of $\Delta_k^F$
is contained in the spectrum of $\Delta_k$ on $H_\mathbb{F}^\ell$,
by \cite[Theorem 1.1]{Polymerakis20b}.
There are two extreme cases:
If $O$ has no parabolic end, then $O$ is convex cocompact.
If $\Omega=\emptyset$, then $\vol O<\infty$ by Bowditch's (F5),
and then (neighborhoods of) the ends of $O$ are cusps, that is,
quotients of horoballs $B$ by parabolic subgroups of $\Gamma$
which act cocompactly on the horospheres $\partial B$.
Cusps are the easiest type of parabolic ends.

In several cases, there are better estimates in the literature than the ones in \cref{thmest2}.
For geometrically finite real hyperbolic manifolds, Mazzeo and Phillips
\cite[Theorem 1.11]{MazzeoPhillips90} calculate the essential spectrum of $\Delta_k$.
Moreover, Bunke and Olbrich determine in \cite{BunkeOlbrich00} the spectral decomposition
for the Hodge-Laplacian on convex cocompact hyperbolic orbifolds.

Carron and Pedon \cite[Theorem B]{CarronPedon04} obtain a lower bound for $\lambda_k(M)$
for hyperbolic manifolds $M=\Gamma\backslash H_{\F}^\ell$,
which are not necessarily geometrically finite,
but for which the critical exponent of $\Gamma$ is at most half the volume growth of $H_{\F}^\ell$.
In the case $\F\neq\R$, there estimate is sharper than our lower bound for $\lambda_0^{\ess}(M)$.
For example, in the complex case,
they show that $\lambda_k(M)\ge(\ell-k)^2$ for $k\ne\ell$ if the critical exponent of $\Gamma$ is at most $\ell$
versus our $\lambda_k^{\ess}(M)\ge(\ell-k-1)^2$ for $k\ne0,\ell-1,\ell,\ell+1,m$ if $M$ is geometrically finite.
(For $k=0,m$, the estimates coincide.)
We suspect that their estimates also hold for geometrically finite orbifolds, but, 
of course, because of the possibility of non-trivial $L^2$-cohomology or small eigenvalues in our context,
only as estimates for the essential spectrum.

Our approach also extends to suborbifolds,
provided their second fundamental form satisfies appropriate restrictions.
We consider properly immersed suborbifolds $P\looparrowright O$ of dimension $n$
together with flat Riemannian vector bundles  $F\to P$.
A result that is easy to state without further preparation is as follows.

\begin{theo}\label{minsub}
If $P$ is a properly immersed minimal suborbifold of a geometrically finite
real hyperbolic orbifold and $\dim_{\R}P=n$, then
\begin{align*}
	\lambda_0^{\ess}(P,F), \lambda_n^{\ess}(P,F) \ge (n-1)^2/4.
\end{align*}
If $P$ is a properly immersed complex suborbifold of a geometrically finite
complex hyperbolic orbifold or a properly immersed quaternion-K\"ahler suborbifold
of a geometrically finite quaternion hyperbolic orbifold and $\dim_{\R}P=n$, then
\begin{align*}
	\lambda_0^{\ess}(P,F), \lambda_n^{\ess}(P,F) \ge n^2/4.
\end{align*}
\end{theo}

Theorems \ref{thmest} -- \ref{minsub} are consequences of a refined version of an inequality
of Donnelly and Xavier \cite[Theorem 2.2]{DonnellyXavier84} and our

\begin{theo}[Main result]\label{mainthm}
For any geometrically finite orbifold $O$ of dimension $m$
with sectional curvature $-b^2\le K_O\le-a^2<0$ and any $\ve>0$,
there are a compact subset $C\subseteq O$ and a vector field $V$ on $O\setminus C$,
which is locally Lipschitz with $|V|=1\pm\ve$, such that,
on a subset $R\subseteq O\setminus C$ of full measure,
the covariant derivative $\nabla V|_x$ exists and its symmetric part has
an eigenvalue in $(-\ve,\ve)$ of multiplicity one with eigenline $\ve$-close to $V_x$
and the other eigenvalues in $(a-\ve,b+\ve)$, for all $x\in R$.
Moreover, if $O=\Gamma\backslash H_\mathbb{F}^\ell$ is hyperbolic, then
\begin{align*}
	\|\nabla V - |R(.,V)V|^{1/2}\|_{R,\infty} < \ve.
\end{align*}
\end{theo}

In the convex cocompact case, the vector field $V$ may be taken to be the gradient field
of the distance function to a convex core of $O$.
Along cusps, we use the gradient field of the associated Busemann function, at least close to infinity.
In general, $V$ arises as a combination of gradient fields of distance functions.

\subsection{Structure of the article}
\label{sustr}
In \cref{secconv},
we collect some results on distance functions to convex sets in Riemannian manifolds.
Most of this is known, but we provide some arguments in places where it seems appropriate.
For later convenience, we also prove \cref{mainthm} in two simple cases.
In \cref{secdoxa},
we derive a refined version of an inequality of Donnelly-Xavier
and discuss some applications in a geometrically simple situation.
In \cref{secgefi}, we extract from \cite{Bowditch95} a short exposition
of the structure of geometrically finite orbifolds.
In \cref{secv}, we prove \cref{mainthm}.
The proof of Theorems \ref{thmest} -- \ref{minsub} is contained in the short \cref{secapp}.
In \cref{symmetry}, we discuss the symmetry of second derivatives of $C^{1,1}$-functions,
an issue, for which we do not know a suitable reference.

\section{Distance to convex subsets}
\label{secconv}

Let $C$ be a closed and convex subset of a complete and connected orbifold $O$.
For the convenience of the reader,
we collect some results about the distance function to $C$.
Two references for this material are \cite{ParkkonenPaulin12,Walter76}. 

For simplicity, we assume throughout that the sectional curvature of $O$ is non-positive.
Then $O=\Gamma\backslash X$,
where $X$ is a simply connected and complete Riemannian manifold of non-positive sectional curvature
and $\Gamma$ a properly discontinuous group of isometries of $X$.
We also assume throughout that the preimage $\tilde C$ of $C$ in $X$ is connected.
Then $\tilde C$ is closed and convex in $X$.
A typical example is the case,
where the sectional curvature of $O$ is negatively pinched and $C$ is the convex core of $O$.

For each $x\in O\setminus C$, there is a unique geodesic $c_x\colon[0,1]\to O$
from $x$ to $C$ such that $\pi(x)=c_x(1)\in C$ satisfies
\begin{align}\label{propi}
	d(x,C) = d(x,\pi(x)).
\end{align}
We call $\pi\colon O\to C$ the \emph{(nearest point) projection to $C$}
and let $f\colon O\to\R$ be the distance function to $C$, $f(x)=d(x,C)$.
Then $f$ is convex and admits Lipschitz constant one.

Now $f(c_x(t))=r(1-t)$ for all $0\le t\le1$, where $r=f(x)$.
Since $f$ admits Lipschitz constant one, the first variation therefore implies that
\begin{align*}
	f(c(s)) - f(x)
	= \la V(x),\dot c(0)\ra s + o(s),
\end{align*}
for any geodesic $c$ from $x$ and sufficiently small $s$, where
\begin{align}\label{gradf}
	V(x) = \frac{-1}r\dot c_x(0) = \frac{-1}{\|\dot c_x(0)\|}\dot c_x(0).
\end{align}
By uniqueness, $\dot c_x(0)$ depends continuously on $x$,
and hence $f$ is $C^1$ on $O\setminus C$ with gradient $\nabla f=V$.

\begin{lem}\label{lemcon}
For $C\subseteq O=\Gamma\backslash X$ as above, we have:
\begin{enumerate}
\item\label{lc2}
the projection $\pi$ admits Lipschitz constant one;
\item\label{lc3}
the distance function $f$ is $C^{1,1}$ on $O\setminus C$
and twice differentiable exactly at the points of $O\setminus C$
at which $\pi$ is differentiable;
\item\label{lc4}
the sublevels $C_r=\{f\le r\}$ of $f$ are convex for all $r>0$.
\end{enumerate}
\end{lem}

\begin{proof}
\eqref{lc2}
Because the preimage of $C$ in $X$ is connected, we may assume that $O=X$.
For $x,y\in O\setminus C$ with $\pi(x)\ne\pi(y)$, we have
\begin{align*}
	\angle_{\pi(x)}(x,\pi(y)), \angle_{\pi(y)}(y,\pi(x)) \ge \pi/2.
\end{align*}
Therefore $d(x,y)\ge d(\pi(x),\pi(y))$ by a standard comparison argument.
The remaining cases follow by analogous arguments.

\eqref{lc3}
The map $\Phi\colon TX\to X\times X$, $\Phi(v)=({\rm foot}(v),\exp(v))$, is a diffeomorphism.
Since
\begin{align}\label{lcpi}
	\nabla f(x) = \frac{-1}{|\Phi^{-1}(x,\pi(x))|} \Phi^{-1}(x,\pi(x))
\end{align}
for any $x\in O\setminus C$ and $\pi$ is Lipschitz continuous,
we conclude that $\nabla f$ is $C^{0,1}$ on $O\setminus C$.
Moreover, \eqref{lcpi} also implies that $f$ is twice differentiable
exactly at the points of $O\setminus C$ at which $\pi$ is differentiable.

\eqref{lc4} follows immediately from the convexity of $f$.
\end{proof}

\begin{cor}\label{lemcon2}
The function $f^2/2$ is $C^{1,1}$ on $O$ with $\nabla(f^2/2)|_x=-\dot c_x(0)$.
\end{cor}

\begin{lem}\label{lemtec}
Let $V=V(s)$ be a curve of tangent vectors on $O$ which is differentiable at $s=0$.
For all $s$, let $\gamma_s$ be the geodesic with initial velocity $V(s)$.
Then $J(t)=\partial\gamma_s(t)/\partial s|_{s=0}$ exists for all $t\in\R$,
and $J$ is the Jacobi field along $\gamma_0$ such that
\begin{align*}
  J'
  = \left.\frac{\nabla}{\partial t}\frac{\partial\gamma}{\partial s}\right|_{s=0}
  = \left.\frac{\nabla}{\partial s}\frac{\partial\gamma}{\partial t}\right|_{s=0}.
\end{align*}
\end{lem}

The point of this lemma is that, in the usual setup, the curve $V$ is assumed to be smooth.
Then $\nabla\partial\gamma/\partial s\partial t=\nabla\partial\gamma/\partial t\partial s$,
and the assertion of \cref{lemtec} follows easily.
Here we assume less, and a little extra thought is needed.
Recall to that end that the pair $(u,\nabla_uV)$ identifies $V'(0)\in T_vTO$,
where $v=V(0)$.

\begin{proof}Let $\pi\colon TO\to O$ be the projection to the foot point
and $(F_t)$ be the geodesic flow of $O$.
Then $\gamma_s(t)=\pi(F_t(V(s)))$ and hence
\begin{align*}
	\left.\frac{\partial\gamma_s(t)}{\partial s}\right|_{s=0}
	= \pi_*F_{t*}(V'(0)).
\end{align*}
Hence we may replace $V$ by a smooth curve with the same derivative at $s=0$
to get that $J(t)$ exists for all $t\in\R$ and that it is equal to the asserted Jacobi field.
\end{proof}

Let $x\in O\setminus C$ be a point,
where the second derivative $\nabla^2f$ exists and is symmetric.
Let $c$ be the unit speed geodesic from $\pi(x)$ through $x=c(r)$, where $r=f(x)$.
For $u\in T_xO$, let $J_u$ be the Jacobi field along $c$ with
\begin{align*}
  J_u(r) = u
  \hspace{2mm}\text{and}\hspace{2mm}
  J_u'(r) = \nabla_u\nabla f.
\end{align*}

\begin{cor}\label{lemtec2}
For all $t>0$, $\nabla^2f$ exists at $c(t)$ and is symmetric;
in fact,
\begin{align*}
  \nabla^2f(J_u(t),J_v(t)) = \la J_u(t),J_v'(t)\ra.
\end{align*}
Furthermore, $\pi_*(J_u(t))=J_u(0)$.
\end{cor}

We also write $J(t)u=J_u(t)$.
Then $J(t)\colon T_xO\to T_{c(t)}O$ is an isomorphism, for all $t>0$.
Furthermore, the covariant derivative of $\nabla f$ satisfies
\begin{align}
	S(t) := \nabla\nabla f|_{c(t)} = J'(t)J(t)^{-1},
\end{align}
by \cref{lemtec2}.
Note that $S$ is a symmetric field of endomorphisms along $c|_{(0,\infty)}$
that satisfies the Riccati equation
\begin{align}\label{riccati}
	S' + S^2 + R_c = 0,
\end{align}
where $R_cu=R(u,\dot c)\dot c$.
Clearly, $\dot c=\nabla f$ belongs to the kernel of $S$.
Therefore we discuss $S$ only on $\dot c^\perp$,
identifying the various $\dot c(t)^\perp$ with $\dot c(0)^\perp$ via parallel translation along $c$.
By \cite[p.\,212]{EschenburgHeintze90}, 
$S$ has the asymptotic behaviour
\begin{align}\label{pq}
	S(t) = \frac{1}{t} P+ Q(t) \hspace{2mm}\text{as $t\to0$},
\end{align}
where $P$ is an orthogonal projection on $\dot c(0)^\perp$
and $Q$ extends continuously to $t=0$, such that $\im P\subseteq\ker Q(0)$.
In terms of $S$, the space of Jacobi fields along $c$ which we consider
is given by the initial conditions
\begin{align}\label{inicon}
	J_v(0) = (1-P)v,\;
	J_v'(0) = Pv + Qv,
\end{align}
where $v\in\dot c(0)^\perp$.
By the convexity of $C$, we have $Q(0)\ge0$.

\begin{lem}\label{curvest}
In the above situation,
\begin{enumerate}
\item\label{cea}
if the sectional curvature of $O$ satisfies $K\le-a^2<0$,
then \[\nabla^2f|_x\ge a\tanh(ar) \hspace{2mm}\text{on $\nabla f(x)^\perp$};\]
\item\label{ceb}
if the sectional curvature of $O$ satisfies $-b^2\le K\le0$,
then \[\nabla^2f|_x\le b\coth(br)\hspace{2mm}\text{on $\nabla f(x)^\perp$}.\]
\end{enumerate}
\end{lem}

\begin{proof}
Let $S_a$ and $S_b$ be solutions of the Riccati equation
along unit speed geodesics $c_a$ and $c_b$ in the real hyperbolic spaces of dimension $m=\dim O$
and constant sectional curvature $-a^2$ and $-b^2$, respectively,
which have the same asymptotic initial conditions at $t=0$ as $S$
with respect to some orthonormal identification
\begin{align*}
	\dot c_a(0)^\perp \cong \dot c(0)^\perp \cong \dot c_b(0)^\perp.
\end{align*}
With respect to an orthonormal basis $(v_i)$ of $\dot c(0)^\perp$ such that $v_1,\dots,v_k$ span $\ker P$
and are eigenvectors of $Q$ with corresponding eigenvalue $\kappa_1,\dots,\kappa_k$
and $v_{k+1},\dots,v_m$ span $\im P$,
we have
\begin{align*}
	S_a(t)v_i =
	\begin{cases}
	a\frac{\sinh(at)+\kappa_i\cosh(at)/a}{\cosh(at)+\kappa_i\sinh(at)/a}v_i
	&\text{for $i\le k$},\\
	a\frac{\cosh(at)}{\sinh(at)}v_i
	&\text{for $i> k$},	
	\end{cases}
\end{align*}
and similarly for $S_b$, substituting $b$ for $a$.
Since $Q(0)\ge0$, we have $\kappa_i\ge0$.
In particular, $S_a$ and $S_b$ are defined for all $t>0$.
By \cite[Theorem]{EschenburgHeintze90}, we have 
\begin{align*}
	S(t) \ge S_a(t)
	\hspace{2mm}\text{respectively}\hspace{2mm} 
	S(t) \le S_b(t)
\end{align*}
for all $t>0$.
This yields the asserted estimates.
\end{proof}

\begin{cor}\label{curvest2}
If $f$ is a Busemann function and
\begin{enumerate}
\item\label{cea2}
the sectional curvature of $O$ satisfies $K\le-a^2<0$,
then \[\nabla^2f|_x\ge a \hspace{2mm}\text{on $\nabla f(x)^\perp$};\]
\item\label{ceb2}
the sectional curvature of $O$ satisfies $-b^2\le K\le0$,
then \[\nabla^2f|_x\le b \hspace{2mm}\text{on $\nabla f(x)^\perp$}.\]
\end{enumerate}
\end{cor}

\begin{proof}
Up to an additive constant, $f$ is the distance function to the horoball $f^{-1}(-\infty,t]$,
for any $t$ close to $-\infty$.
Hence given $x$, the $r$ in \cref{curvest} can be made arbitrarily large 
by choosing $t$ sufficiently close to $-\infty$.
\end{proof}

In the case where $O$ is a quotient of a hyperbolic space $H=H^\ell_\mathbb{F}$,
$m=\ell \dim_\R \mathbb{F}$ and $-4\le K_H\le-1$,
there is a parallel orthogonal decomposition $\dot c^\perp=E_1\oplus E_2$ into eigenspaces of $R_c$,
where $E_2$ is of dimension $\dim_\R{\mathbb F}-1$,
such that $R_cv=-v$ for $v\in E_1$ and $R_cv=-4v$ for all $v\in E_2$.
After identifying the various $\dot c(t)^\perp$ with $E=\dot c(0)^\perp$ via parallel translation,
the Riccati equation \eqref{riccati} becomes
\begin{align}\label{riccatia}
	S' = A^2 - S^2,
\end{align}
where $A=|R_c|^{1/2}$ is a constant symmetric endomorphism of $E$.
We will now discuss \eqref{riccatia} for a symmetric endomorphisms $A$ of a Euclidean vector space $E$
such that
\begin{align}\label{riccatib}
	0 < a \le A \le b.
\end{align}
In the geometric setting, this corresponds to bounds $-b^2\le K\le-a^2<0$ on the sectional curvature.

\begin{lem}\label{riccatil}
Let $S$ be solution of \eqref{riccatia} on $(0,t_+)$ with maximal $t_+$
and asymptotic development \eqref{pq} for $t\to0$, where $Q\ge0$.
Then $t_+=\infty$.
Moreover, for any $\ve>0$, there is a $t_\ve>0$, which only depends on $a$ and $b$, such that
\begin{align*}
	\|S(t)-A\| < e^{-(2a-\ve)(t-t_\ve)}
		\hspace{2mm}\text{for all $t>t_\ve$}.
\end{align*}
\end{lem}

\begin{proof}[Proof of \cref{riccatil}]
By \cite[Theorem]{EschenburgHeintze90}, we have 
\begin{align*}
	a\tanh(at) \le S(t) \le b\coth(bt)
\end{align*}
for all $0<t<t_+$.
Since the space of symmetric endomorphisms of $E$ with given lower and upper bounds is compact,
we conclude that $S(t)$ cannot escape to infinity as $t\to t_+$ and hence that $t_+=\infty$.

Let $v\in E$ be a unit vector.
Then
\begin{align*}
	\la (S-A)v,v\ra' = \la S'v,v\ra
	= \la (A^2-S^2)v,v\ra
	= - \la (S-A)v,(A+S)v\ra.
\end{align*}
Now choose $t_0>0$ such that $\tanh(at_0)\ge 1-\ve/a$, and let $t\ge t_0$.
Suppose that $v$ is a unit vector such that $\la (S(t)-A)v,v\ra$ is minimal or maximal.
Then
\begin{align*}
	(S(t)-A)v = \la (S(t)-A)v,v\ra v
\end{align*}
and therefore
\begin{align*}
	\la (S-A)v,v\ra'(t)
	= - \la (S(t)-A)v,v\ra\la v,(S(t)+A)v\ra.
\end{align*}
Hence if $\max_v\la (S(t)-A)v,v\ra>0$ and $v$ is a corresponding eigenvector,
we have
\begin{align*}
	\la (S-A)v,v\ra'(t) \le -(2a-\ve)\la (S(t)-A)v,v\ra < 0.
\end{align*}
On the other hand, if $\min_v\la (S(t)-A)v,v\ra<0$ and $v$ is a corresponding eigenvector,
then
\begin{align*}
	\la (S-A)v,v\ra'(t)
	\ge -(2a-\ve)\la (S(t)-A)v,v\ra
	> 0.
\end{align*}
Hence the smooth function $s=s(t,v)=\la (S(t)-A)v,v\ra$ on $(0,\infty)\times S_E$,
where $S_E$ denotes the unit sphere of $E$, has the property that
\begin{align*}
	s'(t,v) \le -(2a-\ve) s(t,v)
	\hspace{2mm}\text{and}\hspace{2mm} 
	s'(t,v) \ge (2a-\ve) s(t,v)
\end{align*}
for any $t>0$ and $v\in S_E$ where
\begin{align*}
	s(t,v)(t) = \max_v s(t,v) > 0
	\hspace{2mm}\text{and}\hspace{2mm} 
	s(t,v)(t) = \min_v s(t,v) < 0,	
\end{align*}
respectively.
It follows easily that the Lipschitz functions
\begin{align*}
	s_+(t) = \max_v \{s(t,v),0\}
	\hspace{2mm}\text{and}\hspace{2mm} 
	s_-(t) = \max_v \{-s(t,v),0\}
\end{align*}
satisfy
\begin{align*}
	s_\pm(t) \le e^{-(2a-\ve)(t-t_0)} s_\pm(t_0)
\end{align*}
for all $t>t_0$.
Since $s_+(t_0)$ and $s_-(t_0)$ are bounded in terms of $a$ and $b$,
this yields the asserted inequality with an appropriately chosen $t_\ve$. 
\end{proof}

\begin{cor}\label{riccatic}
If $f$ is a Busemann function on $H_\mathbb{F}^\ell$, then
\begin{align*}
	\nabla^2f = |R(.,\nabla f)\nabla f|^{1/2}.
\end{align*}
\end{cor}

\begin{proof}
The right hand side corresponds to $A$ in \eqref{riccatia} and \eqref{riccatib},
where $a=1$ and $b=2$.
Furthermore, up to an additive constant,
$f$ is the distance function to the horoball $f^{-1}(-\infty,s]$, for any $s\in\R$.
Given any $x\in H_\mathbb{F}^\ell$ and $\delta>0$,
choose $\ve>0$ and then $t>t_\ve$ as in \cref{riccatil} such that $e^{-(2-\ve)(t-t_\ve)}<\delta$.
With $s=f(x)-t$, we get \[\|\nabla^2f|_x - |R(.,\nabla f|_x)\nabla f|_x|^{1/2}\|<\delta,\]
by \cref{riccatil}.
Since this holds for any $x$ and $\delta>0$, the claim follows.
\end{proof}

\begin{proof}[Proof of \cref{mainthm} in two simple cases]
There are two cases, in which \cref{mainthm} is an immediate application of the above results.
Assume first that the limit set $\Lambda\subseteq X_\iota$ of $\Gamma$ is empty or,
equivalently, that $\Gamma$ is finite.
Then $\Gamma$ fixes a point $x\in X$ and, therefore, the distance function to $x$ is $\Gamma$-invariant
and descends to a function $f$ on $O=\Gamma\backslash X$.
The gradient field $V=\nabla f$ then satisfies  the assertions of \cref{mainthm},
by Lemmas \ref{curvest} and \ref{riccatil}.

In a second case, assume that $\Lambda$ consists of exactly one point $x\in X_\iota$.
Then $\Gamma$ fixes $x$ and, hence,
$\Gamma$ leaves Busemann functions centered at $x$ invariant.
Therefore, they descend to functions on $O$.
By Corollaries \ref{curvest2} and \ref{riccatic},
their gradient field satisfies the assertions of \cref{mainthm}.
\end{proof}

\section{Vector fields and Hodge-Laplacian}
\label{secdoxa}

In this section,
we obtain extensions of an inequality of Donnelly and Xavier \cite{DonnellyXavier84},
including previous improvements of their inequality
in \cite{Kasue94,BallmannBruening01}.
In the beginning,
we follow \cite[Section 5]{BallmannBruening01}.  

Let $O$ be a Riemannian orbifold and $V$ be a bounded vector field of class $C^{0,1}$ on $O$.
Then the covariant derivative $\nabla V$ exists almost everywhere.

The field of quadratic forms $Q_V(X)=\la\nabla_XV,X\ra$ only depends on the symmetric part $A$ of $\nabla V$.
Let $\alpha_1,\dots,\alpha_m$ be the eigenvalues of  $A$.
By the variational characterization of eigenvalues of symmetric endomorphisms,
the sums of the $k$ smallest and $k$ largest $\alpha_i=\alpha_i(x)$ depend continuously on $x\in O$,
for any $1\le k\le m$.

Let $F\to O$ be a Riemannian vector bundle with a metric connection
and $\sigma$ be a smooth section of $\Lambda^kO\otimes F$,
that is, a smooth $k$-form on $O$ with values in $F$.
Define a vector field $X$ by the property that
\begin{align}\label{BB1}
	\la X,Y\ra = \la i_V\sigma,i_Y\sigma\ra
\end{align}
for any vector field $Y$.
Then the discussion on \cite[p.\,621]{BallmannBruening01} gives that
\begin{align}\label{BB2}
	\dive X
	+ \la d\sigma,V^\flat\wedge\sigma\ra + \la i_V\sigma,d^*\sigma\ra
	= \la\nabla_V\sigma,\sigma\ra + \sum_i\la i_{AX_i}\sigma,i_{X_i}\sigma\ra
\end{align}
in terms of any local orthonormal frame $(X_i)$ of $T_xO$.
By considering degrees, we get that
\begin{align}\label{BB3}
\begin{split}
	\la d\sigma,V^\flat\wedge\sigma\ra + \la i_V\sigma,d^*\sigma\ra
	&= \la (d+d^*)\sigma,V^\flat\wedge\sigma+i_V\sigma\ra \\
	&= (-1)^k\la D\sigma,\sigma\cdot V\ra
\end{split}
\end{align}
where $D=d+d^*$ denotes the Dirac operator on $\Lambda^*O\otimes F$
and the dot Clifford multiplication from the right on $\Lambda^*O$,
\begin{align*}
	\sigma\cdot V = (-1)^k(V^\flat\wedge\sigma + i_V\sigma).
\end{align*}
If $\sigma$ is compactly supported, we obtain
\begin{equation}\label{BB4}
\begin{split}
	(-1)^k\la D\sigma,\sigma\cdot V\ra_{L^2}
	&= \int_O\sum_i\la i_{AX_i}\sigma,i_{X_i}\sigma\ra -\frac12\int_O\dive V|\sigma|^2 \\
	&= \int_O\sum_i \big\{\la i_{AX_i}\sigma,i_{X_i}\sigma\ra -\frac12\la AX_i,X_i\ra|\sigma|^2\big\},
\end{split}
\end{equation}
where we use that $\nabla_V+\frac12\dive V$ is a skew-symmetric operator
with respect to the $L^2$-inner product.
Clearly, at any point,
the integrand on the right does not depend on the chosen orthonormal frame $(X_i)$ at that point.
In particular, we may choose the $X_i$ pointwise to form an orthonormal eigenbasis for $A$.
Then the right hand side can be evaluated as in \cite[Equation 5.8]{BallmannBruening01} and yields
\begin{align}\label{BB5}
	\la D\sigma,\sigma\cdot V\ra_{L^2}
	= \frac{(-1)^k}2 \int_O\sum_{I,J}
	\left\{ \sum_{i\in I}\alpha_i - \sum_{i\notin I}\alpha_i\right\} |\sigma_{I,J}|^2,
\end{align}
where we write, pointwise,
\begin{align*}
	\sigma = \sum_{I,J} X_I^\flat\otimes\Phi_J
\end{align*}
in terms of the $k$-forms $X_I^\flat=X_{i_1}^\flat\wedge\dots\wedge X_{i_k}^\flat$,
for all $1\le i_1<\dots<i_k\le m$,
and an orthonormal frame $(\Phi_J)$ of $F$.

\begin{rem}
Modifying the definition of $X$ in \eqref{BB1} and defining a vector field $X'$ by requiring that,
for any vector field $Y$,
\begin{align*}
	\la X',Y\ra = \la V^\flat\wedge\sigma,Y^\flat\wedge\sigma\ra
\end{align*}
also leads to \eqref{BB5} and does not give further information.
\end{rem}

For a symmetric endomorphism $A$ or a symmetric bilinear form $S$
on a Euclidean vector space $E$ of dimension $m$, let
\begin{align}\label{trace}
	\tr A = \sum \la Au_i,u_i\ra
	\hspace{3mm}\text{and}\hspace{3mm}
	\tr S = \sum S(u_i,u_i) 
\end{align}
be the traces of $A$ and $S$ on $E$,
where the $u_i$ form an orthonormal basis of $E$.
For all $0\le k\le m$, set
\begin{align}\label{tracek}
	\delta_k(A) = \min_L (\tr A|_L - \tr A|_{L^\perp})
	\hspace{3mm}\text{and}\hspace{3mm}
	\delta_k(S) = \min_L (\tr S|_L - \tr S|_{L^\perp}),
\end{align}
where $L$ runs over all subspaces of $E$ of dimension $m-k$.
For any $0\le k\le m$, define now a continuous function
\begin{align}\label{deltak}
	\delta_k = \delta_k(x) = \delta_k(A_x)
\end{align}
on $O$.
Clearly, $\dive V=\delta_0$.
By the variational characterization of eigenvalues of $A$, we have
\begin{align}\label{deltak2}
	\delta_k(x) = \min_I\left\{\sum_{i\notin I}\alpha_i(x) - \sum_{i\in I}\alpha_i(x)\right\},
\end{align}
where $I$ runs over all subsets of $\{1,\dots,m\}$ with $k$ elements.
Clearly, the infimum at $x$ is achieved by $I$
if $\alpha_i(x)\ge\alpha_j(x)$ for any $i\in I$ and $j\notin I$.

In the case where $\delta_k\ge0$, we get from \eqref{BB5} that
\begin{align}\label{BB7}
	|\la D\sigma,\sigma\cdot V\ra_{L^2}|
	= \frac12 \int_O\sum_{I,J}
	\left| \sum_{i\notin I}\alpha_i - \sum_{i\in I}\alpha_i\right| |\sigma_{I,J}|^2
\end{align}
for any compactly supported smooth form $\sigma$ on $O$ with values in $F$ of degree $k$ or $m-k$.
With this, we arrive at the following somewhat generalized form of \cite[Theorem 5.3]{BallmannBruening01}.

\begin{thm}\label{BB8}
If $\delta_k\ge0$, then
\begin{align*}
	\|D\sigma\|_{L^2}\|\sigma\|_{L^2}\|V\|_\infty
	\ge \frac1{2} \int_O \delta_k |\sigma|^2
\end{align*}
for any compactly supported smooth form $\sigma$ on $O$
with values in $F$ of degree $k$ or $m-k$.
\end{thm}

Assume now that $F$ is flat.
Then $d^2=(d^*)^2=0$ and hence the (twisted) Hodge-Laplacian $\Delta=D^*D$
leaves the degree of forms invariant.

\begin{cor}\label{BB9}
Suppose that $F$ is flat and that $d_k=\inf\delta_k>0$.
Then
\begin{align*}
	\la\Delta\sigma,\sigma\ra_{L^2}\|V\|_\infty^2
	= \|D\sigma\|_{L^2}^2\|V\|_\infty^2
	\ge \frac{d_k^2}{4} \|\sigma\|_{L^2}^2
\end{align*}
for any compactly supported smooth form $\sigma$ on $O$ with values in $F$ of degree $k$ or $m-k$.
In other words, $d_k^2/4\|V\|_\infty^2$ is a lower bound for the spectrum
of the Friedrichs extension of $\Delta$ on $C^\infty_c(\Lambda^kO\otimes F)$.
\end{cor}

\begin{rem}\label{comdel}
Recall that $\Delta$ on $C^\infty_c(\Lambda^kO\otimes F)$ is essentially self-adjoint
if $O$ is complete, so that, in this case,
the Friedrichs extension of $\Delta$ coincides with the closure of $\Delta$.
\end{rem}

\begin{rem}[Essential spectrum]\label{essu}
We note that
\begin{align*}
	\lambda_k^{\ess}(O,F) = \sup_C\lambda_k(O\setminus C,F|_{O\setminus C}),
\end{align*}
where $C$ runs over compact subsets of $O$.
This is, e.g., a consequence of \cite[Theorem A.14]{BallmannPolymerakis20},
whose proof in \cite{BallmannPolymerakis20}, which is for manifolds, also applies to orbifolds;
compare also with \cite[Proposition 4.8]{BallmannPolymerakis21}.
In particular, if $U$ is an open neighborhood of infinity in a Riemannian orbifold $O$,
we may let $U$ take over the role of $O$ in the above discussion
to conclude that $d_k^2/4\|V\|_\infty^2$ is a lower bound of the essential spectrum of $\Delta$,
where $V$ is a $C^{0,1}$-vector field on $U$ with the corresponding $d_k>0$.
\end{rem}

\subsection{Suborbifolds}
\label{suscalsu}
Let $P\looparrowright O$ be an isometrically immersed suborbifold of dimension $n\le m$.
Then the above discussion applies to the component $V^{\top}$ of $V$ tangential to $P$.
Denoting the Levi-Civita connection of $P$ by $\nabla^{\top}$, we have
\begin{align*}
	\la\nabla^{\top}_XV^{\top},Y\ra
	&= \la\nabla_XV^{\top},Y\ra \\
	&= \la\nabla_XV,Y\ra - \la\nabla_XV^\perp,Y\ra \\
	&= \la\nabla_XV,Y\ra + \la S(X,Y),V^\perp\ra,
\end{align*}
for all vector fields $X,Y$ tangential to $P$,
where $S$ denotes the second fundamental form of $P$.
Since $S$ is symmetric, we conclude that
\begin{align}\label{abara}
	\la BX,Y\ra
	= \la AX,Y\ra + \la S(X,Y),V^\perp\ra
\end{align}
for all vector fields $X,Y$ tangential to $P$,
where $A$ and $B$ denote the symmetric parts of $\nabla V$ and $\nabla^{\top}V^{\top}$, respectively.

For any $0\le k\le n$, define continuous functions
\begin{align}
	\delta_k = \delta_k(x) &= \inf_L \{\tr(B_x|_L)-\tr(B_x|_{L^\perp}\} \label{badek} \\
	\delta_k' = \delta_k'(x) &= \inf_L \{\tr(A_x|_L)-\tr(A_x|_{L^\perp}\} \label{deprik} \\
	\gamma_k = \gamma_k(x) &= \inf_L \{\tr(S_x^V|_L)-\tr(S_x^V|_{L^\perp}\} \label{gamk}
\end{align}
on $P$, where $L$ runs over all subspaces of $T_xP$ of dimension $m-k$
and
\begin{align*}
	S^V = S^V(X,Y) = \la S(X,Y),V\ra = \la S(X,Y),V^\perp\ra
\end{align*}
denotes the second fundamental form of $P$ in the direction of $V^\perp$.
If $H$ denotes the mean curvature vector field of $P$ and $h=|H|$, then
\begin{align}\label{ghv}
    \gamma_0=\la H,V^\perp\ra
    \hspace{3mm}\text{and}\hspace{3mm}
    |\gamma_0|\le h|V^\perp|.
\end{align}
By \eqref{abara},
\begin{align}\label{dedegak}
	\delta_k \ge \delta_k' + \gamma_k.
\end{align}
In particular, the analogs of \cref{BB8} and \cref{BB9} hold if
\begin{align}\label{dedegak2}
	\delta_k \ge \delta_k' + \gamma_k \ge 0
	\hspace{3mm}\text{and}\hspace{3mm}
    d_k \ge d_k' + h_k > 0,
\end{align}
respectively, where we let
\begin{align}\label{ddhk}
     d_k = \inf\delta_k, \hspace{3mm}
     d_k' = \inf\delta_k',
     \hspace{3mm}\text{and}\hspace{3mm}
     h_k = \inf\gamma_k.
\end{align}
Clearly, the comments in Remarks \ref{comdel} and \ref{essu} also apply.

\begin{rem}
In an analogous way, one may discuss Riemannian submersions,
where the vector field $V$ would be the horizontal lift of a vector field on the base.
However, since we do not have interesting new results along these lines,
we refrain from pursuing this and refer the interested reader to \cite{CavalcanteManfio18,Polymerakis20}
for results in this context.
\end{rem}

\subsection{Some simple applications}
\label{susim}
We discuss now three examples, in which the above results apply.
The computations in Examples \ref{pifix} and \ref{hyfix} will also be used later on
in the proofs of Theorems \ref{thmest} -- \ref{minsub}.

Let $X$ be a complete and simply connected Riemannian manifold of dimension $m$
with negative sectional curvature, $K_X<0$.
Denote by $X_\iota$ the \emph{ideal boundary} of $X$.
Let $\Gamma$ be a properly discontinuous group of isometries of $X$ that fixes a point $x\in X_\iota$.
Then the vector field $V$ of unit vectors in $X$ pointing away from $x$
is the gradient field of the Busemann functions centered at $x$
and is $C^{1}$ (see \cite{HeintzeImHof77}) with symmetric covariant derivative $A=\nabla V$.
It is invariant under $\Gamma$ and is therefore well defined on $O=\Gamma\backslash X$.
Recall that $A$ has eigenvalue $0$ in the direction of $V$.

Let $P\looparrowright O$ be an isometrically immersed suborbifold of dimension $n$
and $F\to P$ a flat Riemannian vector bundle over $P$.
The case $P=O$ is not excluded.

\begin{exa}[Strictly negative curvature]\label{stricfix}
Suppose that $K_X\le-a^2$ for some $a>0$.
Then $A\ge0$ and $A\ge a$ perpendicular to $V$, by \cref{curvest2}.\ref{cea2}.

For $x\in P$, the trace of $A$ on $T_xP$ is at least $(n-1)a$, by what we just said.
Therefore $\delta_0' \ge d_0' \ge (n-1)a$ and hence, by \eqref{dedegak2},
\begin{align}\label{bessamon}
	\lambda_0(P,F), \lambda_n(P,F)
	\ge \frac14((n-1)a+h_0)^2
\end{align}
if $-h_0<(n-1)a$.
In the case where $P=O=X$ and $F=\R$,
we have $h_0=0$ and \eqref{bessamon} is due to McKean \cite[p.\ 360]{McKean70}.
In the case where $O=X$ and $F=\R$, \eqref{bessamon} improves the corresponding
estimate \cite[Corollary 4.4]{BessaMontenegro03} of Bessa-Montenegro slightly.
\end{exa}

\begin{exa}[Pinched negative curvature]\label{pifix}
Suppose that $-b^2\le K_X\le-a^2$ for some $b>a>0$.
Then $A\ge0$ and $a\le A\le b$ perpendicular to $V$, by \cref{curvest2}.

Let $x\in P$ and $L\subseteq T_xP$ be a subspace of dimension $n-k$.
Then the trace of $A$ is at least $(n-k-1)a$ on $L$
and at most $kb$ on the perpendicular complement $L^\perp$ of $L$ in $T_xP$.
Hence
\begin{align}\label{pifixd}
	\delta_k' \ge d_k' \ge (n-k-1)a - kb 
\end{align}
and therefore, by \eqref{dedegak2},
\begin{align}\label{pifixl}
	\lambda_k(P,F), \lambda_{n-k}(P,F)
	\ge \frac14((n-k-1)a-kb+h_k)^2
\end{align}
for all $0\le k\le n$ with $kb-h_k<(n-k-1)a$.
In the case where $P=O=X$ and $F=\R$, we have $h_k=0$ and \eqref{pifixl} sharpens
the corresponding estimate \cite[Theorem 3.2]{DonnellyXavier84} of Donnelly-Xavier.
By \cite[Example 5.5]{BallmannBruening01},
\eqref{pifixl} and the associated estimate in \cref{BB8} are optimal in the case $m=n$.
\end{exa}

\begin{exa}[The hyperbolic case]\label{hyfix}
If $X=H_{\F}^\ell$, where $m=\ell d$, then $A=|R_X(.,V)V|^{1/2}$, by \cref{riccatic}.
In particular, $A$ has eigenvalue $2$ of multiplicity $d-1$, $1$ of multiplicity $(\ell-1)d$,
and $0$ of multiplicity $1$.

Let $x\in O$ and $L\subseteq T_xO$ be a subspace of dimension $m-k$.
If $k\le d-1$, then the trace of $A$ is at least $m+d-2k-2$ on $L$ and at most $2k$ on $L^\perp$.
If $k\ge d-1$, then the trace of $A$ is at least $m-k-1$ on $L$ and at most $k+d-1$ on $L^\perp$ of $L$.
Hence
\begin{align}\label{hyfixd}
	\delta_k' \ge d_k' \ge
	\begin{cases}
	m + d - 2 - 4k &\text{for $0\le k\le d-1$} \\
	m - d - 2k &\text{for $d-1\le k\le n$}
	\end{cases}
\end{align}
and therefore, by \eqref{dedegak2},
\begin{align}\label{hyfixl}
	\lambda_k(O,F), \lambda_{m-k}(O,F)
	\ge \frac14
	\begin{cases}
	((m + d - 2 - 4k + h_k)^2 \\
	((m - d - 2k + h_k)^2
	\end{cases}
\end{align}
for all $k\le d-1$ with $4k-h_k<m+d-2$ respectively all $k\ge d-1$ with $2k-h_k<m-d$.

Consider now $P\looparrowright O$ and $F\to P$ as above.
It would be nice to have inequalities analogous to \eqref{pifixd} and \eqref{pifixl}.
In the case $d=1$, that is, $\F=\R$,
this is no problem since it corresponds to the case $a=b$ in  \cref{pifix}.
However, for $\F\ne\R$, the eigenspace of $A$ for the eigenvalue $2$ has dimension $d-1>0$.
Depending now on $n$ and $k$, there are then various possibilities
of how the eigenspace intersects a subspace $L$ of a tangent space of $P$ of dimension $n-k$.
We have a satisfying answer for this issue only in the case $k=0$
for complex $P$ in complex hyperbolic $O$
and quaternion-K\"ahler $P$ in quaternion hyperbolic $O$; see below.
We start, however, with a non-optimal general inequality.

Let $x\in P$ and $L\subseteq T_xP$ be a subspace of dimension $n-k$.
Then, neglecting possible better contributions of eigenvalues $2$ of $A_x$ in $L$,
the trace of $A_x$ is at least $n-k-1$ on $L$
and is at most $k+(d-1)\wedge k$ on the perpendicular complement $L^\perp$ of $L$ in $T_xP$.
Hence
\begin{align}\label{hyfixd2}
	\delta_k' \ge d_k' \ge n - 1 - 2k - (d-1)\wedge k 
\end{align}
and therefore, by \eqref{dedegak2},
\begin{align}\label{hyfixl2}
	\lambda_k(P,F), \lambda_{n-k}(P,F)
	\ge \frac14(n -  1 - 2k - (d-1)\wedge k  + h_k)^2
\end{align}
for all $0\le k\le n$ with $2k+(d-1)\wedge k-h_k<n-1$.

The case $d=1$, that is, $\F=\R$,
corresponds to the case $a=b$ in \cref{pifix} and is optimal.
However, we can do better than in \eqref{hyfixl2} for $\F=\C$ and $k=0$
in the case where $P$ is a complex suborbifold of $O$
and for $\F=\H$ and $k=0$ in the case where $P$ is a quaternion-K\"ahler suborbifold of $O$.

To that end, we consider the case $\F=\C$ first, that is,
we let $O$ be a complex hyperbolic orbifold with complex structure $J$.
Then
\begin{align}\label{casech}
	A + J^{-1}AJ = 2\id
\end{align}
since $JV$ spans the field of eigenspaces of $A$ for the eigenvalue $2$
and the field of eigenspaces of $A$ for the eigenvalue $1$ is $J$-invariant.

Let $P\looparrowright O$ be a complex suborbifold and $x\in P$.
Then $T_xP$ is a complex subspace of $T_xO$, hence $\delta_0\ge n$ by \eqref{casech},
and therefore
\begin{align}\label{casech2}
	\lambda_0(P,F), \lambda_n(P,F) \ge \frac14n^2,
\end{align}
where we use that complex suborbifolds of K\"ahler orbifolds are minimal.

We let now $\F=\H$ and consider a quaternion hyperbolic orbifold $O$.
We let $x\in O$ and choose a compatible quaternion structure $IJ=K$ on $T_xO$.
Then
\begin{multline}\label{caseqh}
	A_x + I^{-1}A_xI + J^{-1}A_xJ + K^{-1}A_xK \\ =
	\begin{cases}
	6\id &\text{on the quaternion span of $V(x)$,} \\
	4\id &\text{on the eigenspace for the eigenvalue $1$ of $A_x$,}
	\end{cases}
\end{multline}
since $IV(x)$, $JV(x)$, and $KV(x)$ span the eigenspace of $A_x$ for the eigenvalue $2$
and the eigenspace of $A_x$ for the eigenvalue $1$ is a quaternion subspace.

Let $P\looparrowright O$ be a quaternion-K\"ahler suborbifold and $x\in P$.
Then $T_xP$ is a quaternion subspace of $T_xO$, hence $\delta_0\ge n$ by \eqref{caseqh},
and therefore
\begin{align}\label{caseqh2}
	\lambda_0(P,F), \lambda_n(P,F) \ge \frac14n^2,
\end{align}
where we use that quaternion-K\"ahler suborbifolds of quaternion-K\"ahler orbifolds are totally geodesic;
see \cite[Theorem 5]{Gray69}.
\end{exa}

There are applications similar to the ones in the above examples
in the case where $\Gamma$ fixes a point in the interior of $X$.
However, the global existence of the vector field $V$ makes the above examples particularly simple
and allows, without further ado, to obtain lower bounds for the spectrum instead of the essential spectrum.

\section{Geometrically finite orbifolds}
\label{secgefi}

Let $X$ be a complete and simply connected Riemannian manifold
with sectional curvature $-b^2\le K=K_X\le-a^2<0$.
Denote by $X_\iota$ the \emph{ideal boundary} of $X$ and set $X_c=X\cup X_\iota$,
the compactification of $X$ with respect to the cone topology \cite{EberleinONeill73}.
Let $\Gamma$ be a discrete group of isometries of $X$
and $\Lambda=\Lambda_\Gamma$ the \emph{limit set of $\Gamma$},
a closed and $\Gamma$-invariant subset of $X_\iota$.
Then $\Omega=\Omega_\Gamma=X_\iota\setminus\Lambda$
is called the \emph{domain of discontinuity of $\Gamma$}.
In fact, the action of $\Gamma$ on $X\cup\Omega$ is properly discontinuous
and $M_c(\Gamma)=\Gamma\backslash(X\cup\Omega)$ is a topological orbifold.
Recall that $O=\Gamma\backslash X$ is called \emph{convex cocompact} if $M_c(\Gamma)$ is compact.
More generally and following \cite[Definition on p.\,265]{Bowditch95},
we say that $O$ is \emph{geometrically finite}
if $M_c(\Gamma)$ has at most finitely many ends, and each end of $M_c(\Gamma)$ is parabolic.

To describe the notions of geometric finiteness and parabolic ends in the way we need it,
we need some more details about the geometry of $X$.
Terminology and results are mostly from \cite{Bowditch95}.

For any two points $x,y\in X_c$, we denote by $[x,y]$ the geodesic connecting them.
For $x\ne y$, we also use the notation $(x,y]$, $[x,y)$, and $(x,y)$
to exclude the respective endpoint or both of them from $[x,y]$.

For any closed subset $Q\subseteq X_c$, we denote by
\begin{enumerate}
\item\label{na}
$N_r(Q)$ the smallest closed subset of $X_c$ containing $Q$
and all points $x\in X$ of distance at most $r$ from $Q\cap X$;
\item\label{ja}
$JQ$ the union of geodesics $[x,y]$,
where $x,y$ run through pairs of points in $Q$;
\item\label{ha}
$HQ$ the closed convex hull of $Q$, that is,
the smallest closed and convex subset of $X_c$ containing $Q$.
\end{enumerate}
Clearly, $JQ$ is a closed subset of $X_c$ and $JQ\subseteq HQ$.
By \cite[Lemma 2.2.1]{Bowditch95},  $JQ$ is $\lambda$-\emph{quasiconvex}
in the sense that $JJQ\subseteq N_\lambda(JQ)$, where $\lambda=\lambda(a)>0$.

For $x\in X$, $y\ne x$ in $X_c$, and $\theta>0$, we define the \emph{(closed) cone}
\begin{align}\label{cone}
	C(x,y,\theta)
	= \{ z\in X_c \setminus \{x\} \mid \angle_x(y,z)\le\theta \} \cup \{x\}.
\end{align}
By \cite[Proposition 2.5.1]{Bowditch95} (following \cite{Anderson83}),
there is a $\theta_0=\theta_0(a/b)>0$ such that the convex hull
\begin{align}\label{covcon}
	HC(x,y,\theta) \subseteq C(x,y,\pi/2)
\end{align}
for all $x\in X$, $y\in X_c\setminus\{x\}$, and $0<\theta\le\theta_0$.
By \cite[Corollary 2.5.3]{Bowditch95}, we have
\begin{align}\label{covinf}
	HQ \cap X_\iota = Q \cap X_\iota
\end{align}
for any closed subset $Q\subseteq X_c$
and by \cite[Proposition 2.5.4]{Bowditch95},
that there is an $r=r(\lambda)>0$ such that
\begin{align}\label{covhu}
	HQ \subseteq N_r(Q)
\end{align}
for any $\lambda$-quasiconvex closed subset $Q\subseteq X_c$.

The most important case is $Q=\Lambda$, the limit set of $\Gamma$.
Since $\Lambda$ is invariant under $\Gamma$, 
the same holds for $J\Lambda$ and $H\Lambda$,
and $H\Lambda$ is a closed and convex subset of $X_c$.

If $|\Lambda|\ge2$, then $H\Lambda\cap X\ne\emptyset$.
Then, for any point $x\in X_c\setminus H\Lambda$,
there is a unique point $y=\pi_{H\Lambda}(x)\in H\Lambda\cap X$
such that $\angle_y(x,H\Lambda)\ge\pi/2$.
For $x\in H\Lambda\cap X$, we let $\pi_{H\Lambda}(x)=x$.
The \emph{(orthogonal) projection}
\begin{align}
	\pi_{H\Lambda} \colon X_c\setminus \Lambda = X \cup \Omega \to H\Lambda\cap X
\end{align}
is $\Gamma$-invariant and admits Lipschitz constant one on $X$.
Clearly, we may retract $X\cup\Omega$ and $X$
along the connecting geodesics $[x,\pi_{H\Lambda}(x)]$ onto $H\Lambda\cap X$,
and this deformation retraction is also $\Gamma$-invariant.
As a result, we obtain that the \emph{convex core}
\begin{align}\label{covcor}
	C = C_\Gamma = \Gamma\backslash(H\Lambda\cap X)
\end{align}
of $M_c(\Gamma)$ is a deformation retract of $M_c(\Gamma)$ and of $O$,
where the retraction is also along the corresponding geodesics.

We now come to the definition of parabolic ends.
We say that a group $G$ of isometries of $X$ is \emph{parabolic} if
\begin{enumerate}
\item[(P1)]
$G$ has a unique fix point $p\in X_c$, and $p$ belongs to $X_\iota$;
\item[(P2)]
$G$ leaves Busemann functions centered at $p$ invariant.
\end{enumerate}
Since finite groups of isometries of $X$ fix a point in $X$,
parabolic groups are infinite.

Let $G$ be a discrete parabolic group of isometries of $X$ with fix point $p\in X_\iota$.
Then $\Omega_G=X_\iota\setminus\{p\}$
and $M_c(G)=G\backslash(X\cup\Omega_G)$ has one end, \emph{the one coming from $p$}:
For $x\in X$, define
\begin{align}\label{cpx}
	C(x) = C_p(x) = \cap_{g\in G}HC(gx,p,\theta_0)\setminus\{p\},
\end{align}
a $G$-invariant closed and convex subset of $X_c\setminus\{p\}=X\cup\Omega_G$.
Moreover, the complement of $G\backslash C(x)$ in $M_c(G)$ is relatively compact.
Thus $G\backslash C(x)$ is a neighborhood of the--unique--end of $M_c(G)$.
Clearly, if $x=x_0\in X$ and $(x_n)$ is a sequence of points on $[x,p)$ converging to $p$,
then $\cap C(x_n)=\emptyset$.
In fact, the $G\backslash C(x_n)$ constitute a basis of neighborhoods of the end of $M_c(G)$.
Compare with \cite[255:14--21]{Bowditch95}.

We say that a point $p\in X_\iota$ is a \emph{parabolic point of $\Gamma$}
if the stabilizer $G_p$ of $p$ in $\Gamma$ is a parabolic group such that,
for $x\in X$ sufficiently close to $p$, the set $C_p(x)$,
defined with respect to $G=G_p$, is \emph{precisely invariant};
that is, 
\begin{align}\label{precise}
	g\in \Gamma
	\hspace{2mm}\text{and}\hspace{2mm}
	gC_p(x) \cap C_p(x) \ne \emptyset
	\hspace{3mm}\Longrightarrow\hspace{3mm} 15
	g\in G_p, 
\end{align}
and then $gC_p(x)=C_p(x)$.
Then $G_p\backslash C_p(x)$ embeds into $M_c(\Gamma)$
and the unique end of $G_p\backslash C_p(x)$ is an end of $O$
and, by definition, a \emph{parabolic end};
compare with \cite[264:1--20]{Bowditch95}, where the phrasing is somewhat different.

Finally, we say that $O$ is \emph{geometrically finite} if $M_c(\Gamma)$ has at most finitely many ends,
and each end of $M_c(\Gamma)$ is parabolic.

\section{Construction of a vector field $V$}
\label{secv}

Let $O=\Gamma\backslash X$ be a geometrically finite orbifold, $\Lambda$ the limit set of $\Gamma$
and $H\Lambda\subset X_c$ the closed convex hull of $\Lambda$; cf.\ \cref{secgefi}.
Recall that we proved \cref{mainthm} at the end of \cref{secconv} in the case where $|\Lambda|\le1$.
Therefore we assume in this section that $|\Lambda|\ge2$.
Then $H\Lambda\cap X\ne\emptyset$, and the distance function $f=f(x)=d(x,H\Lambda)$
is well-defined on $X$ and is $C^1$ on $X\setminus H\Lambda$.

\subsection{Construction of a vector field $V_0$}
\label{sucv}
For $x\in X\setminus H\Lambda$,
the negative gradient $-\nabla f|_x$ is the unit vector in the direction of the point $\pi_{H\Lambda}(x)$
in $H\Lambda$ nearest to $x$.
Identifying the tangent bundle $TX$ via the exponential map with $X\times X$
and using that $\pi_{H\Lambda}$ admits Lipschitz constant one,
we see that the vector field $V_0=\nabla f$ is $C^{0,1}$
and, as a consequence, that $f$ is $C^{1,1}$.
Since $H\Lambda$ is $\Gamma$-invariant, $V_0$ is $\Gamma$-invariant
and induces therefore a vector field on $\Gamma\backslash X$
outside of the convex core $C=\Gamma\backslash(H\Lambda\cap X)$ of $O$.

At any point $x\in X\setminus H\Lambda$ in the set of full measure
at which $f$ is twice differentiable with symmetric second derivative $\nabla\nabla f$,
$V_0$ is differentiable with differential $\nabla_{\dot\sigma(0)}V_0=J'(r)$,
where $J$ is the Jacobi field along the unit speed geodesic $c$
from $\pi_{H\Lambda}(x)$ through $x=c(r)$ with $J(r)=\dot\sigma(0)$,
where $J$ corresponds to the variation of $c=c_0$ by the geodesics $c_s$
from $\pi_{H\Lambda}(\sigma(s))$ through $\sigma(s)$.
By convexity, we have $\la J(0),J'(0)\ra\ge0$ for any such $J$,
and therefore \cref{curvest} implies that, for any $\ve>0$, 
\begin{align}\label{est1}
	a\tanh(ar)-\ve \le \nabla V_0|_x \le b\coth(br)+\ve
\end{align}
on $V_0(x)^\perp$ if $r=f(x)$ is sufficiently large and $x$ is as above.
Then the symmetric endomorphism $\nabla V_0|_x$
has eigenvalue $0$ in the direction of $V_0(x)$
and eigenvalues in $[a\tanh(ar)-\ve,b\coth(br)+\ve]$ on $V_0(x)^\perp$.
In particular, if $r$ is sufficiently large and $x$ is as above,
then the eigenvalues of $\nabla V_0|_x$ are in $[a-\ve,b+\ve]$,
except for the eigenvalue $0$ in the direction of $V_0(x)$.

In the case where $X=H_\mathbb{F}^\ell$ is a hyperbolic space with $\max K_X=-1$
and $x$ is as above, we get from \cref{riccatil} that
\begin{align}\label{est2}
	\|A_x - \nabla V_0|_x \| \le \ve
\end{align}
if $r=f(x)$ is sufficiently large, where $A_x=|R(.,V_0(x))V_0(x)|^{1/2}$.
Here the conclusion is that, except for the eigenvalue $0$ in the direction of $V_0(x)$,
$\nabla V_0|_x$ has $d-1$ eigenvalues close to $2$ and $d(\ell-1)$ eigenvalues close to $1$,
where $d=\dim_{\R}\mathbb{F}$.

\subsection{Construction of vector fields $V_p$}
\label{sucvp}
Busemann functions $f$ centered at a parabolic point $p$
are invariant under $G_p$ and differ by a constant only.
Since Busemann functions are $C^2$ with uniformly bounded second derivative \cite{HeintzeImHof77},
in fact smooth in the case where $X$ is a hyperbolic space,
the gradient vector field $V=V_p=\nabla f$ is well defined and $C^1$ with uniformly bounded derivative.
It has constant norm one and is invariant under $G_p$,
hence defines a vector field on the neighborhood $G\backslash C(x)$
of the corresponding end of $M_c(\Gamma)$.
Recall that, for all $x\in X$,
\begin{align}\label{est3}
	a\tanh(ar) \le \nabla V_p|_x \le b\coth(br)
\end{align}
on $V_0(x)^\perp$.
Furthermore, if $X=H_F^n$, then
\begin{align}\label{est4}
	\nabla V_p|_x = A_x
\end{align}
for all $x\in X$.
Thus the conclusions from \cref{sucv} hold also for the $\nabla V_p$.

\subsection{Joining $V_0$ to the $V_p$}
\label{sucvvp}
Our aim is now to combine $V_0$ with the various vector fields $V_p$
to a single vector field $V$ of norm approximately one with appropriate estimates of $\nabla V$.
Outside of the neighborhoods $G_p\backslash C_p(x)$ of the ends of $M_c(\Gamma)$,
we let $V=V_0$.
Inside of the $G_p\backslash C_p(x)$,
we change from $V_0$ to the corresponding $V_p$ as described in what follows.

Given any discrete group $\Gamma$ of isometries of $X$, $\ve>0$, and $x\in X$,
let $\Gamma_\ve(x)$ be the subgroup of $\Gamma$ generated by the elements $g\in\Gamma$
with $d(x,gx)<\ve$ and set
\begin{align}\label{teps}
	T_\ve(\Gamma) = \{ x\in X \mid |\Gamma_\ve(x)|=\infty\}.
\end{align}
We are going to use these notions for the group $\Gamma$ in question and its parabolic subgroups $G_p$.
Recall that, by the Margulis lemma, $\Gamma_\ve(x)$ is almost nilpotent if $0<\ve<\ve(m,\kappa)$.

In what follows, let $p$ and $G_p$ be as above and $0<\ve<\ve(m,\kappa)$.
Let $Q\subseteq X_\iota\setminus\{p\}$ be a closed and $G_p$-invariant subset
such that $G_p\backslash Q$ is compact, and set $H=H(Q\cup\{p\})\setminus\{p\}$.
By \cite[Lemma 4.10]{Bowditch95}, given $q\in X_\iota\setminus\{p\}$,
there are a horoball $B$ in $X$ with center $p$ and an $r>0$ such that
\begin{align}\label{bow1}
	H \cap T_\ve(G_p) \subseteq H \cap B \subseteq N_r(G_p(q,p)) \cap B.
\end{align}
Conversely, by \cite[Lemma 4.11]{Bowditch95}, for any $q\in X_\iota\setminus\{p\}$ and $r>0$,
there is a horoball $B$ in $X$ with center $p$ such that
\begin{align}\label{bow2}
	N_r(G_p(q,p)) \cap B \subseteq T_\ve(G_p).
\end{align}
Furthermore, by \cite[Proposition 4.12]{Bowditch95}, for $x\in(q,p)$ sufficiently close to $p$,
\begin{align}\label{bow3}
	C_p(x) \subseteq \pi_H^{-1}(H \cap T_\ve(G_p)),
\end{align}
where $\pi_H \colon X_c\setminus\{p\} \to H$ is the projection.

By \cite[Lemma 5.11]{Bowditch95}, $Q=\Lambda\setminus\{p\}$ satisfies the assumptions on $Q$ above.
Then $H=H\Lambda\setminus\{p\}$.

\begin{lem}\label{prepar}
Let $p\in X_\iota$ be a parabolic point of $\Gamma$ with stabilizer $G_p\subseteq\Gamma$.
Let $0<\ve<\ve(m,\kappa)$, $r>0$, and $q\in X_\iota\setminus\{p\}$.
Then there is an $s>0$ such that,
for all $x\in X\setminus N_s(G_p(q,p))$ with $\pi_H(x)\in H\cap N_r(G_p(q,p))$,
there is a $g\in G$ such that $\angle_x(\pi_H(x),y)<\ve$ for all $y\in g(q,p)$.
\end{lem}

\begin{proof}
Let $g\in G_p$.
Then, if $s>s(\ve)$, we have $\angle_x(g(q,p))<\ve/2$.
On the other hand, if $g(q,p)$ contains a point $z$ with $d(z,\pi_H(x))\le r$,
then $\angle_x(z,\pi_H(x))<\ve/2$ if $s>s(r)$.
Combining the two estimates, we obtain the assertion.
\end{proof}

\begin{proof}[Proof of \cref{mainthm}]
The searched for set $C$ will be a large neighborhood of the thick part of the convex core of $\Gamma\backslash X$.
The technical problems in the discussion are due to the parabolic ends of $\Gamma\backslash X$.

Choose a set $R\subseteq X_\iota$ of representatives modulo $\Gamma$
of the set $P$ of parabolic points of $\Gamma$,
and note that $R$ is finite.
Given $0<\ve<\ve(m,\kappa)$, there is a point $x_p\in X$ for each $p\in R$
such that each $C_p(x_p)$ is precisely invariant,
that the $C_p(x_p)$ are pairwise disjoint, and, by passing to a smaller $\ve>0$ if necessary, that
\begin{align*}
	T_\ve(\Gamma)\cap C_p(x_p)
	= T_\ve(G_p)\cap C_p(x_p)
	= T_\ve(G_p)
\end{align*}
for all $p\in R$.
We extend these choices $\Gamma$-equivariantly to $P$. 
The statements corresponding to the above then hold for all $p\in P$.
Moreover, since $\Gamma$ is geometrically finite,
the $\Gamma$-invariant set
\begin{align*}
	H_0 = H\Lambda\setminus\cup_{p\in P} \mathring C_p(x_p)
\end{align*}
is compact modulo $\Gamma$.

Let $p\in R$ and $q\in Q=\Lambda\setminus\{p\}$ be the point with $x_p\in(q,p)$.
Choose a horoball $B=B_0$ with center $p$ and an $r>0$ which satisfy \eqref{bow1},
and let $x_0=(q,p)\cap\partial B_0$.
Choose a point $x_1\in(x_0,p)$ with $d(x_0,x_1)\ge1$ such that $x_p\in(q,x_1]$
and such that $C_p(x_1)$ satisfies \eqref{bow3}.
Let $B_1\subseteq B_0$ be a horoball with center $p$ contained in $C_p(x_1)$
and $B_2\subseteq B_1$ be the horoball with center $p$ such that $d(\partial B_1,B_2)=1$.
Then any point in $B_1\setminus\mathring B_2$ projects under $\pi_H$ to $H\cap T_\ve(G_p)$,
which is contained in $N_r(G_p(q,p))\cap B_0$.
Choose $s=s(r,\ve)$ according to \cref{prepar}.
Then $|V_0-V_p|<\ve$ on $C_p(x_1)\setminus N_s(G_p(p,q))$.
Furthermore, $(C_p(x_1)\cap N_s(G_p(p,q)))\setminus \mathring B_2$ is compact modulo $G_p$
and contains $B_1\setminus\mathring B_2$.
Now we change from $V_0$ to $V_p$ on $(C_p(x_1)\setminus N_s(G_p(p,q))) \cup\mathring B_2$,
using the Lipschitz one function $\psi$ which is one outside of $B_1$ and zero inside of $B_2$,
by setting
\begin{align*}
	V = \psi V_0 + (1-\psi) V_p.
\end{align*}
It has covariant derivative
\begin{align*}
	\nabla V
	= \psi\nabla V_0 + (1-\psi)\nabla V_p + \nabla\psi\otimes(V_0-V_p).
\end{align*}
Since $|\nabla\psi|\le1$, the norm of the error term
\begin{align*}
	\nabla\psi\otimes(V_0-V_p)
\end{align*}
is bounded by $|V_0-V_p|<\ve$ on $B_1\setminus (N_s(G_p(q,p))\cup\mathring B_2)$
and vanishes otherwise.
Since $V_0$ and $V_p$ are $G_p$-invariant on $C_p(x_1)$,
$V$ pushes down to a vector field on the corresponding part of $G_p\backslash C_p(x_1)$.
It has the property that, given $\ve>0$, there is a compact subset $C$ of $\Gamma\backslash X$
such that $|V|=1\pm\ve$ and that $\nabla V$ has the asserted properties.
\end{proof}

\section{Estimating the essential spectrum}
\label{secapp}

We now apply \cref{mainthm},
using the results from \cref{secdoxa}.
Let $O=\Gamma\backslash X$ be a geometrically finite orbifold of dimension $m$
with sectional curvature $-b^2\le K_O\le -a^2<0$, where $0<a\le b$.
Let $P\looparrowright O$ be a properly immersed suborbifold of dimension $n\le m$,
endowed with the induced Riemannian metric
and $F\to P$ be a flat Riemannian vector bundle.

\begin{thm}\label{thmpin}
If $(n-k-1)a-kb+h_k>0$, then 
\begin{align*}
	\lambda_k^{\ess}(P,F), \lambda_{n-k}^{\ess}(P,F)
	\ge \frac14((n-k-1)a-kb+h_k)^2
\end{align*}
\end{thm}

\begin{proof}
For any $\ve>0$, choose $C$ and $V$ as in \cref{mainthm}.
Then, on $O\setminus C$, $|V|=1\pm\ve$ and the corresponding $\delta_k$ is at least $(n-k-1)a-kb-n\ve$.
Since $P$ is properly immersed, $P\cap C$ is compact in $P$.
Hence
\begin{align*}
	\frac14\frac{((n-k-1)a-kb+h_k-n\ve)^2}{(1+\ve)^2}
\end{align*}
is a lower bound for the essential spectrum of $P$, by \cref{BB9},
the characterization of the essential spectrum in \cref{essu},
and the computations in \cref{pifix}.
\end{proof}

\cref{thmest} corresponds to the case $P=O$ in \cref{thmpin}.
In a similar fashion, Theorems \ref{thmest2} and \ref{minsub}
are consequences of \cref{mainthm}, \cref{BB9}, \cref{essu}, and \cref{hyfix}.

\appendix

\section{The symmetry of second derivatives}
\label{symmetry}

Say that $x\in\R^m$ is a $2$-Lebesgue point of a map $f\colon\R^m\to\R^n$ if,
for any orthonormal $u,v\in\R^m$ tangent to coordinate directions,
\begin{align*}
	\lim_{r\to0}\frac1{r^2}\int_0^r\int_0^r |f(x+su+tv)-f(x)| = 0.
\end{align*}
Together with the Fubini theorem,
the Lebesgue differentiation theorem implies that almost any point of $\R^m$
is a $2$-Lebesgue point of $f$ if $f$ is locally integrable;
compare with \cite[Section 2.1.4]{GiaquintaModica09}.

\begin{lem}\label{lemsym}
For any $f\in C^{1,1}(\R^m,\R^n)$,
$d^2f(x)$ is symmetric at each $2$-Lebesgue point $x$ of the map $d^2f$.
\end{lem}

In view of our needs, we assume that $f$ is $C^{1,1}$,
but somewhat weaker assumptions would also be sufficient.

\begin{proof}[Proof of \cref{lemsym}]
For $u,v\in\R^m$, we have
\begin{align*}
	f(x+ru+rv) - f(x&+ru) - f(x+rv) + f(x) \\
	&= \int_0^r \{df(x+ru+tv)-df(x+tv)\}v\dt \\
	&= \int_0^r\int_0^r d^2f(x+su+tv)(u,v)\ds\dt \\
	&= I_{u,v}(r) + r^2d^2f(x)(u,v),
\end{align*}
where we note, for the penultimate equality, that $df$ is $C^{0,1}$ and where
\begin{align*}
	I_{u,v}(r) = \int_0^r\int_0^r \{ d^2f(x+su+tv)-d^2f(x)\}(u,v)\ds\dt.
\end{align*}
Interchanging the roles of $u$ and $v$, we obtain
\begin{align*}
	f(x+ru+rv) - f(x&+ru) - f(x+rv) + f(x) \\
	&= \int_0^r\int_0^r d^2f(x+su+tv)(v,u)\dt\ds \\
	&= I_{v,u}(r) + r^2d^2f(x)(v,u).
\end{align*}
If $x$ is a $2$-Lebesgue point of $d^2f$ and $u,v$ are orthonormal and tangent to coordinate directions,
then we have
\begin{align*}
	\lim_{r\to0}\frac1{r^2}  I_{u,v}(r)
	= \lim_{r\to0}\frac1{r^2} I_{v,u}(r)
	= 0.
\end{align*}
Therefore, by the above computations,
\begin{align*}
	|d^2f(x)(u,v)-d^2f(x)(v,u)|
	&\le \lim_{r\to0}\frac1{r^2} |I_{u,v}(r)-I_{v,u}(r)| = 0.
\end{align*}
Hence $d^2f(x)$ is symmetric.
\end{proof}

\begin{cor}
For any $k\ge1$ and $f\in C^{k,1}(\R^m,\R^n)$,
$d^{k+1}f(x)$ is symmetric at each $2$-Lebesgue point $x$ of the map $d^{k+1}f$.
\end{cor}

\bibliographystyle{amsplain}
\bibliography{References}
\end{document}